\documentclass[twoside]{kybernetika} 
\usepackage{amsmath}
\usepackage{amsthm}
\usepackage{hyperref}

\usepackage{amssymb}
\usepackage{verbatim}
\usepackage{multido}
\usepackage{color}
\usepackage{rotating}
\usepackage{latexsym}
\usepackage{dsfont}
\usepackage{enumerate}
\usepackage{enumitem}

\newtheorem{theorem}{Theorem}
\newtheorem{proposition}[theorem]{Proposition}
\newtheorem{lemma}[theorem]{Lemma}
\newtheorem{corollary}[theorem]{Corollary}

\newtheorem{remark}[theorem]{Remark}

\theoremstyle{definition}
\newtheorem{definition}[theorem]{Definition}
\newtheorem{example}[theorem]{Example}

\newtheorem{problem}[theorem]{Problem}

\newcommand{\X}{{\mathcal{X}}}

\newcommand{\Acal}{{\mathcal{A}}}
\newcommand{\Pcal}{{\mathcal{P}}}
\newcommand{\oP}{\overline{\mathcal{P}}}
\newcommand{\E}{{\mathcal{E}}}
\newcommand{\Ecal}{\mathcal{E}}
\newcommand{\Fcal}{\mathcal{F}}
\newcommand{\Xcal}{\mathcal{X}}
\newcommand{\oE}{\overline{\mathcal{E}}}
\newcommand{\F}{{\mathcal{F}}}
\newcommand{\N}{{\mathcal{N}}}
\newcommand{\R}{{\mathbb{R}}}
\newcommand{\Mcal}{{\mathcal{M}}}
\newcommand{\Y}{{\mathcal{Y}}}
\newcommand{\Ycal}{{\mathcal{Y}}}
\newcommand{\Z}{{\mathcal{Z}}}
\newcommand{\supp}{{\operatorname{supp}}}

\newcommand{\Mixt}{\operatorname{Mixt}}
\newcommand{\conv}{{\operatorname{conv}}}
\newcommand{\dotcup}{\ensuremath{\mathaccent\cdot\cup}}

\usepackage{wasysym}
\usepackage{bbm}

\newcommand{\mysymbolpentagon}{\pentagon}
\newcommand{\mymathbb}{\mathbb}

\begin{document}
\pagestyle{myheadings}

\setcounter{page}{20}
\title{Mixture decompositions of exponential \newline  
families using a decomposition \newline of their sample spaces}

\author{
Guido F. Mont\'ufar  
}

\contact{Guido}{Mont\'ufar}{Department of Mathematics, Pennsylvania State University, University Park, PA 16802. U\,.S\,.A.}{gfm10@psu.edu}

\markboth{G. Mont\'ufar} {Mixture decompositions of exponential
families}
\maketitle

\begin{abstract}
We study the problem of finding the smallest $m$ such that every
element of an exponential family can be written as a mixture of
$m$ elements of another exponential family. We propose an approach
based on coverings and packings of the face lattice of the
corresponding convex support polytopes and results from coding
theory. We show that $m=q^{N-1}$ is the smallest number for which
any distribution of $N$ $q$-ary variables can be written as
mixture of $m$ independent $q$-ary variables. Furthermore, we show
that any distribution of $N$ binary variables is a mixture of $m =
2^{N-(k+1)}(1+ 1/(2^k-1))$ elements of the $k$-interaction
exponential family.
\end{abstract}



\section{Introduction}

The $m$-mixture of a set of probability distributions $\mathcal{M}$ is the set of all possible convex combinations of $m$ of its points:
\[
\Mixt^m(\mathcal{M})\!:=\!\Big\{\sum_{j=1}^m
\alpha_j p_j \;\Big|\; p_j\!\in\! \mathcal{M}, \;\alpha_j
\!\geq\!0 \;\text{ for } j\in\{1,\ldots,m\},   \text{ and }
\sum_{j=1}^m\alpha_j \!=\!1 \Big\}\;.\label{mixturemodel}
\]
The numbers $\alpha_j\in\mymathbb{R}_{\geq0}$ are called mixture weights and the summands $p_j$  mixture components.
There is an abundant literature on mixture models, see~\cite{Bishop:2006,Lindsay1995,McLachlan:2000,titterington:1985}.
They arise within probabilistic models that involve latent variables, see for instance~\cite{Montufar2011,MonRauh2011}.
An exponential family on a finite set $\Xcal$, with sufficient statistics $A\in\R^{d\times \X}$ and reference measure $\nu\in\R^\X_{>0}$,  is the set of probability distributions $p_\theta$, parametrized by $\theta\in\R^d$, of the form
\[
p_\theta(x)  = \frac{1}{Z_\theta}\nu(x) \exp(\theta^\top  A_x)
\quad \forall\, x\in\X\;,
\]
where $A_x$, $x\in\X$ are the columns of $A$, and $Z_\theta=\sum_{y\in\X} \nu(y) \exp(\theta^\top A_y)$ is the partition function.
See Section~\ref{sect:preliminaries} for details and~\cite{amari:00a,Brown86:Fundamentals_of_Exponential_Families,Efron:1978} for standard references.
We consider the following problem:
\begin{problem}\label{problem1}
{Given two exponential families $\E$ and $\E'$ on a finite set $\X$,
find the smallest natural number $m=m(\E,\E')$,  if there is any, for which $\operatorname{Mixt}^m({\E}) \supseteq {\E'}$.}
\end{problem}
\noindent
We propose a general approach based on coverings and packings of support sets of probability distributions,  combinatorics of convex polytopes, and results from coding theory.
We give explicit solutions when $\Ecal$ is the independence model of $N$ finite valued random variables, or a $k$-interaction exponential family, expressed in terms of the number of variables and the cardinality of their state spaces.
When $\E'$ is equal to the convex hull of $\Ecal$, for instance equal to the set $\Pcal$ of strictly positive probability distributions on $\X$, then $m(\Ecal,\Ecal')$ is the {\em Carath\'eodory number} of $\Ecal$.
We address Problem~\ref{problem1} for closures of exponential families as well.
The closure of a statistical model $\Mcal$, in the standard topology of the real valued functions, is denoted by $\overline{\Mcal}$.
When $\overline{\Ecal}$ is the set of product distributions of $N$ random variables, then $m(\overline{\Ecal},\Ecal')$ is the maximal non-negative outer-product rank of the $N$-way tables of probabilities described by $\Ecal'$.
Problem~\ref{problem1} can be thought of as a tensor decomposition problem. \\

The problem of representing probability distributions as mixtures
of specific models has a long record. A renowned result in this
direction is de~Finetti's theorem, which states that exchangeable
sequences of Bernoulli (i.\,e., binary) variables, are mixtures of
independent and identically distributed Bernoulli variables,
see~\cite{diaconis:1977,Kingman1978}. In general, the expressive
power of mixture models is not satisfactorily understood.
Until recently it was a long standing problem whether the $m$-mixture of the set of probability distributions of $n$ independent binary variables had the dimension expected from parameter counting, which is $\min \{n\cdot m +(m-1), 2^n-1\}$.
M. Catalisano, A. Geramita, and A. Gimigliano~\cite{Catalisano2011} proved that this mixture model indeed has the expected dimension for any combination of $m$ and $n$, except for $n=4$ and $m=3$ when the dimension is smaller.
In connection with this, the identifiability of parameters of mixtures of independent binary variables has been treated, for example in~\cite{BC2011}.
For mixtures of independent non-binary variables the dimension and parameter identifiability problems are largely unsettled. \\

When ${\overline{\Ecal}}$ is the set of probability distributions of two independent variables with values in $\X_1$ and $\X_2$, respectively, it is  known that $\Mixt^m(\overline{\E})$ equals the set $\overline{\Pcal}$ of all possible probability distributions (on $\X_1\times\X_2$) as soon as $m\geq \min\{|\X_1|,|\X_2|\}$, see~\cite{Gilula1979,Settimi:1998}.
This is to say that every non-negative $k\times l$ matrix can be written as the sum of at most $\min\{k,l\}$ non-negative rank-one matrices.
If $|\X_1|,|\X_2|>2$, it is known that $\Mixt^2(\overline{\E})\neq \overline{\Pcal}$, see~\cite{Gilula1979}.
We generalize these results in Theorem~\ref{mixtqary}:
\begin{center}
 {\em \begin{minipage}{\textwidth} The smallest $m$ for which any probability distribution on $\{1,\ldots,q\}^N$ can be written as the mixture of $m$ product distributions, is $q^{N-1}$ (when $q$ is a prime power).\end{minipage}}
\end{center}
The result $q^{N-1}$ is larger than expected from na\"ive parameter counting. In particular, the $m$-mixture model of $N\geq5$ independent binary variables has the same dimension as $\Pcal$ whenever $m\geq {2^N}/{(N+1)}$.
The $m$-mixture of a $k$-interaction model can be viewed as a system of stochastic units including higher-order interactions and a hidden $m$-valued variable. We show (Theorem~\ref{decompfromcubeplus}):
\begin{center}
{\em \begin{minipage}{\textwidth} The smallest $m$ for which any
probability distribution on $\{0,1\}^N$ can be represented as the
mixture of $m$ distributions from the $k$-interaction model is at
most $2^{N-(k+1)}(1+
\tfrac{1}{(2^k-1)})$.\end{minipage}}
\end{center}
We provide similar, however weaker, results when the variables are not binary, but take values in arbitrary finite sets. We also give a bound on the smallest number of mixtures of independent binary distributions needed to represent $k$-interaction models. \\

Our proofs are based on comparison of the support sets of probability distributions contained in the closures of different exponential families.
The support of a probability distribution $p$ is the set $\supp(p):=\{x\in\Xcal\colon p(x)>0\}$.
Combinatorial aspects of support sets of closures of exponential families have been studied in~\cite{kahle:2008,kahle:2006,kahle:2009,RKA10:Support_Sets_and_Or_Mat}.
We add to this analysis and put forward the analysis of a special type of support sets:
\begin{definition}
Given a set of probability distributions $\mathcal{M}$ on a finite
set $\X$ we call $\Y\subseteq\X$ an {\em $S$-set} of $\mathcal{M}$
iff every probability distribution $p$ with support
$\supp(p)\subseteq\Y$, is contained in $\overline{\mathcal{M}}$.
\end{definition}
\noindent
The ``$S$'' in this definition stands for ``support'' and
``simplex'', considering that the set of probability distributions
$p$ with  $\supp(p)\subseteq\Y$ is a simplex (the convex hull of
affinely independent points in Euclidian space). A probability
distribution $p$ can be decomposed as a mixture of $m$ probability
distributions from the closure of $\Mcal$ whenever the support of
$p$ is contained in the union of $m$ $S$-sets of $\Mcal$. This
gives rise to the problem: {\em Given an exponential family on
$\X$, find the smallest possible collection of $S$-sets that
covers $\X$.}

In Section~\ref{sect:preliminaries} we review basics of exponential families, their support sets, and convex supports.
Section~\ref{sect:ssets} formalizes our approach and discusses $S$-sets of exponential families.
Section~\ref{section:mixtureskmodel}
treats coverings and packings using support sets of independence models and $k$-interaction families, and contains our solutions to Problem~\ref{problem1} for these models.
Technical proofs are displaced to the Appendix.

\section{Exponential Families and Convex Supports}\label{sect:preliminaries}

We consider a system of $N\in\mymathbb{N}$ random variables $X_i$ with values in finite sets $\X_i$  for $i\in[N]:=\{1,\ldots,N\}$. The joint {\em sample space} of this system is $\X:=\times_{i=1}^N\X_i$.
The probability distributions with support $\Y\subseteq\X$ are denoted by $\mathcal{P}(\Y)$, or just by $\Pcal$ if $\Y=\X$ is clear.
The closure $\overline{\Pcal}(\Y)$  is the set of all probability distributions $p$ with $\supp(p)\subseteq\Y$ and is called the probability simplex on $\Y$.
The variable $X_i$ is called $q$-ary when $|\X_i|=q$.
For a subset of indices $\lambda\subseteq [N]$, $x_\lambda$ denotes an element of $\times_{i\in\lambda}\X_i$, or the natural restriction of some $x\in\X$ to the coordinates $i\in\lambda$.
The expression $[x_\lambda]$ represents a {\em cylinder set} of dimension $(N-|\lambda|)$, defined as the set of all $y\in\X$ with $y_\lambda=x_\lambda$.
In the binary case  $\X=\{0,1\}^N$ the cylinder sets and the (sets of vertices of) faces of the $N$-dimensional unit cube $[0,1]^N$ are in natural correspondence.  \\

Consider a strictly positive function $\nu$ on $\X$, and a linear subspace $V$ of the space $\R^\X$ of real valued functions on $\X$.
The exponential family $\E_{\nu,V}$ is defined as the image of $V\to\mathcal{P}\subset\mymathbb{R}^\X\;;\; f\mapsto\nu \exp(f)/\sum_{x\in \X}\nu(x)\exp(f(x))$.
For simplicity we set $\nu\equiv1$ and omit the subscript, as the results contained in this paper hold for any strictly positive $\nu$.
A matrix $A\in\mymathbb{R}^{d\times\X}$ with row span $V$ is
called a {\em sufficient statistics} of $\mathcal{E}_V$.  The rows
of $A$ are functions on $\X$ called {\em observables}. Denoting
the columns by $A_x, x\in\X$, the probability distributions in
$\mathcal{E}_V$ can be written as $p_\theta(x) =
\frac{1}{Z_\theta}\exp(\theta^\top  A_x)\; \forall\, x\in\X\;
\forall\, \theta\in\R^d$, where $Z_\theta:=\sum_y \exp(\theta^\top
A_y)$. For simplicity we always denote a sufficient statistics by
$A$ and the corresponding exponential family by $\E$. The
parametrization given above depends on $A$, but $\E$ itself only
depends on $V$ (modulo the constant functions). We assume, without
loss of generality, that $\mathds{1}:=(1,\ldots,1)$ is a row of
$A$. The map $\theta\mapsto p_\theta$ is bijective and $\Ecal$ has
dimension $\operatorname{rk}(A)-1$ exactly when the rows of $A$
(including $\mathds{1}$), are linearly independent, see for
instance~\cite{amari:00a}. The elements of an exponential family
$\E$ are strictly positive.
The closure $\overline{\Ecal}$ includes probability distributions with support strictly contained in $\X$. \\

Given a collection of sets $\Delta\subseteq 2^{[N]}$, the {\em
hierarchical model} $\Ecal_{\Delta}$ is the exponential family
defined by $V_\Delta:=\{\sum_{\lambda\in\Delta} f_\lambda\colon
f_\lambda\in\mymathbb{R}^{\X} \text{ with }
f_\lambda(x_\lambda,x_{[N]\setminus\lambda})=f_\lambda(x_\lambda,\tilde{x}_{[N]\setminus\lambda})\newline\forall\,
x,\tilde x\in\X,\; \forall\, \lambda\in\Delta \}$.  The {\em
$k$-interaction exponential family} $\mathcal{E}^k$ is the
hierarchical model $\E_{{\Delta_k}}$ with $\Delta_k:=\{\lambda
\subseteq [N]: |\lambda|\leq k\}$. The special case $\E^1$ is
called {\em independence model}. The independence model consists
of all strictly positive independent distributions, or product
distributions, of the variables $X_1,\ldots,X_N$. There is a
natural hierarchy of nested models
$\mathcal{E}^1\subset\mathcal{E}^2\subset\cdots\subset\mathcal{E}^N=\mathcal{P}$,
see details
in~\cite{Amari:1999,AyKnauf06:Maximizing_Multiinformation}. The
dimension of $\Ecal_\Delta$ is
 $\dim(\E_\Delta)= \sum_{\lambda\in\Delta}\prod_{i\in\lambda}(|\X_i|-1) -1$, see~\cite{Hosten2002}.
The binary $k$-interaction model has dimension $\dim(\E_{N,\text{bin}}^k)=\sum_{i=1}^k{N\choose i}$.
The sufficient statistics of any binary hierarchical model $\Ecal_\Delta$ can be chosen as  $A=(\sigma_{\lambda,x})_{\lambda\in\Delta,x\in\{0,1\}^N}$, where
\[
\sigma_{\lambda,x}:=(-1)^{|\supp(x)\cap\lambda|}\quad \forall\,
x\in\{0,1\}^N\quad\forall\, \lambda\in2^{[N]}\;.
\]
The rows of $\sigma=(\sigma_{\lambda,x})_{\lambda\in2^{[N]},x\in\{0,1\}^N}$ with labels $\lambda$ from an inclusion complete set $\Delta\subseteq 2^{[N]}$ are an orthogonal basis of $V_\Delta\subseteq\R^\X$, $\Xcal=\{0,1\}^N$.  In particular, $\sigma$ is a Hadamard matrix. \\

The {\em convex support} of $\Ecal$, as realized from a sufficient statistics $A$,
is the image of the {\em moment map}, $\pi \colon \overline{\mathcal{P}}\to\mymathbb{R}^d\,;\; p \mapsto A\cdot p$. This is the following convex polytope (the convex hull of finitely many points in Euclidian space):
\[
Q:= \operatorname{conv}\{A_x\}_{x\in\X} \;.
\]
The moment map $\pi$ defines a homeomorphism of $\overline{\E}$ and $Q$,  and $A\cdot p$ is called the {\em expectation parameter} vector of the point $p\in \overline{\Ecal}$, see~\cite{amari:00a,Efron:1978} and further details in the Appendix.
A {\em face} of the polytope $Q$ is the intersection of $Q$ with a hyperplane in $\mymathbb{R}^{d}$ such that all points of $Q$ lie on one of the closed halfspaces defined through that hyperplane.
In particular, $Q$ is a face of itself.
The dimension of a face $F$ is defined as the dimension of its affine hull $\dim (F) :=\dim \operatorname{aff} (F)$.
The {\em combinatorial type} of $Q$ is the set of all its faces, denoted by $\F(Q)$, together with the partial order of inclusion.
For any $0\leq g\leq \dim(Q)-1$ the union of $g$-dimensional faces $\cup_{F\in \F(Q): \dim(F) = g} F$ contains all vertices of $Q$~\cite[Theorem~15.1.2]{henk:1997}.
Any nonsingular affine transformation of a polytope preserves its combinatorial type~\cite[Theorem~3.2.3]{Gruenbaum:2003}. In turn, the combinatorial type of $Q$ depends only on the row span of $A$ (modulo the constant functions). \\

A set $\Y\subseteq\X$ is called a {\em facial} set of the
exponential family $\Ecal$ iff  $\Y=\{x\in\X\colon A_x\in F\}$ for
some face $F$ of $Q$. The set of all facial sets of $\Ecal$ is
denoted by $\F(\E)\subseteq2^\X$. It is well known that $\F(\E)$
and $\F(Q)$ are in one-to-one correspondence (see, for
example~\cite{geiger:2006, RKA10:Support_Sets_and_Or_Mat}): A set
$\Y\subseteq\X$ is the support of a distribution
$p\in\overline{\E}$ if and only if $\Y$ is a facial set of
$\Ecal$.

\begin{example}
\label{ex1}
Consider the set of strictly positive product distributions of two binary variables, $p(x_1,x_2) = p_1(x_1) p_2(x_2)$ for all $(x_1,x_2)\in\{0,1\}^2$, where $p_1$ and $p_2$  are strictly positive distributions on $\{0,1\}$. This is an exponential family $\Ecal^1$ with sufficient statistics
\[
A=\underset{(00) \;(01) \;(10)\; (11)}{\begin{pmatrix}1 &1&1&1 \\ 1&1&0&0 \\ 1&0&1&0\end{pmatrix}}\;,
\]
whereby $p_1(x_1) p_2(x_2)=\frac1Z \exp(\theta ^\top A_{(x_1
x_2)})$ when $\exp(\theta_2) = p_1(0)/p_1(1)$ and 
$\exp(\theta_3)=p_2(0)/p_2(1)$. The parameter $\theta_1$ is
irrelevant. The convex support $Q_1=\conv\{A_x\}_{x\in\{0,1\}^2}$
is a square. See Figure~\ref{figure2fam}. The two-mixture
$\operatorname{Mixt}^2(\overline{\mathcal{E}^1})$ is the union of
all  line segments $\{\alpha p + (1-\alpha)q \colon
\alpha\in[0,1]\}$ connecting pairs
$p,q\in\overline{\mathcal{E}^1}$. The support sets of
distributions in $\overline{\mathcal{E}^1}$ are $\{0,1\}^2$, the
pairs $\{(00),(01)\}$, $\{(01),(11)\}$, $\{(11),(10)\}$,
$\{(10),(00)\}$, and all the points $\{(00)\}, \{(01)\},\{(10)\}$,
$\{(11)\}$. All support sets are $S$-sets, except for $\{0,1\}^2$.
Figure~\ref{figure2fam} reveals that every point in the
probability simplex is a mixture of two distributions with
supports in the $S$-sets $\{(0,0),(0,1)\}$ and $\{(1,0),(1,1)\}$.
These two $S$-sets cover the entire sample space $\{0,1\}^2$.
\end{example}

\begin{figure}
\begin{center}
\setlength{\unitlength}{7cm}
\begin{picture}(.4,.75)(.2,0)
\put(0,0){\includegraphics[width=3.8 cm]{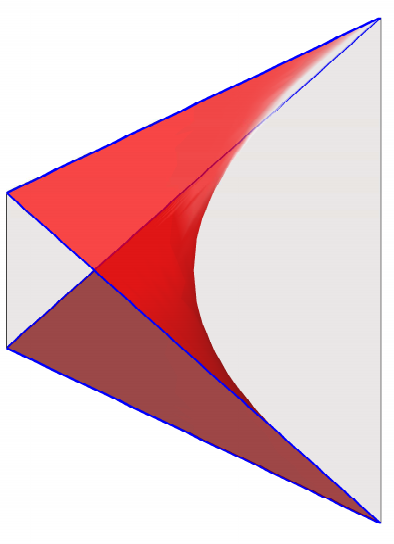}}
\put(.32,.35){$\mathcal{E}^1$} \put(-0.03,.54){$\delta_{(00)}$}
\put(.53,.699){$\delta_{(01)}$} \put(-0.03,.205){$\delta_{(11)}$}
\put(.53,.0){$\delta_{(10)}$}
\end{picture}
\begin{picture}(.4,.75)(.4,0)
\put(0.6,0.2){\includegraphics[width=2.2 cm]{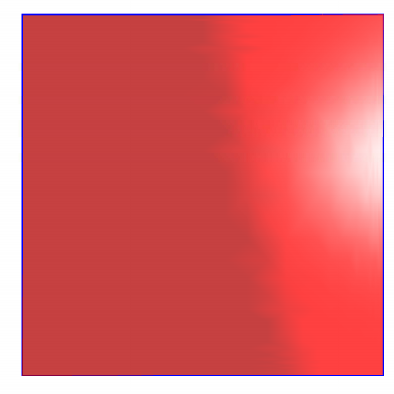}}
\put(.74,.35){$Q_1$} \put(.6,.534){$A_{(01)}$}
\put(.83,.534){$A_{(00)}$} \put(.6,.17){$A_{(11)}$}
\put(.83,.17){$A_{(10)}$}
\end{picture}
\end{center}
\caption{Left: The set of probability distributions of two
independent binary variables $\overline{\mathcal{E}^1}$ (red
surface) within the three-dimensional simplex of probability
distributions on $\{0,1\}^2$. The vertices of the probability
simplex are the point measures $\delta_x$, $x\in\{0,1\}^2$
(distributions with a single support point $\{x\}$). Right: The
convex support of $\Ecal^1$ realized as the convex hull of the
sufficient statistics $A$ from Example~\ref{ex1}.
}\label{figure2fam}
\end{figure}
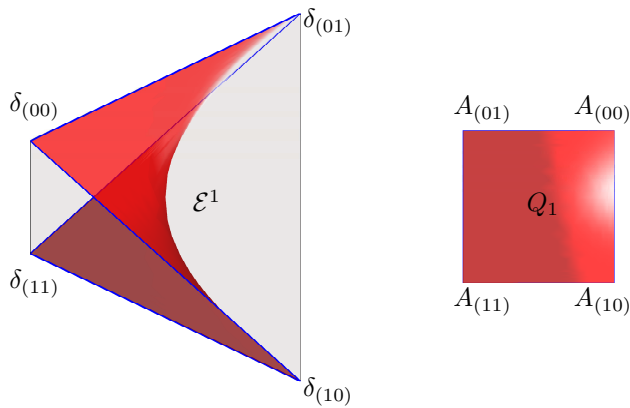

\section{{\em S}-sets of Exponential Families}\label{sect:ssets}

We assess the expressive power of mixture models, comparing the support sets of distributions from different models.
In this section we formalize the idea, and relate $S$-sets of exponential families to simplex faces of their convex supports. \\

Given an exponential family $\E$ on $\X$ we consider the following function, which gives the minimal cardinality of a facial packing  of any set $\Z\subseteq\X$:
\[
\kappa_\E^f:\;2^\X\to\mymathbb{N}\;;\;\Z\mapsto\min \{n \in \mathbb{N} \colon \exists\, \Y_1,\ldots, \Y_n \in\F(\E) \text{ with } \cup_i\Y_i =\Z\}\;.
\]
We set $\kappa_\E^f(\Z)=\infty$ if there does not exist a facial packing of $\Z$.
All $\Y_i$ in this definition are required to be subsets of $\Z$.
For many exponential families, including hierarchical models (with $\cup_{\lambda\in\Delta}\lambda=[N]$), every $\{x\}$ is a facial set. In particular $\kappa_k^f:=\kappa_{\E^k}^f<\infty$ for all  $k>0$.
We also consider the smallest number of $S$-sets that cover $\Z$, which is the following function:
\[
\kappa_\E^s:\;2^\X\to\mymathbb{N}\;;\;\Z\mapsto\min \{n\in\mathbb{N} \colon \exists\, \Y_1,\ldots,\Y_n \text{ $S$-sets with } \cup_i\Y_i\supseteq\Z\}\;,
\]
whereby we set $\kappa_\E^s(\Z)=\infty$ if there does not exist an $S$-set covering of $\Z$.
If $\kappa$ $S$-sets cover $\X$, then at most $\kappa$ $S$-sets are needed for packing any $\mathcal{Z}\subseteq\X$, because any subset of an $S$-set is an $S$-set.
We abbreviate $\kappa_\E^s(\X)$ with $\kappa_\E^s$.
Finally, given two exponential families $\E$ and $\E'$, we consider the maximum of  $\kappa_\E^f$ restricted to the facial sets of $\E'$:
\[
\kappa_{\E,\E'}^f:=\max_{\Z\in\F(\E')}\kappa_{\E}^f(\Z)\;.
\]
The functions $\kappa_\Ecal^f$ and $\kappa_\Ecal^s$ can be defined for any model $\Mcal\subseteq\overline{\mathcal{P}}$ in the place of the exponential family $\Ecal$ by simply replacing ``facial sets''  with ``support sets of distributions within $\overline{\mathcal{M}}$.''
We have the following:

\begin{lemma}
\label{mixturedecompk}\mbox{}
Consider two exponential families $\mathcal{E},\mathcal{E}'\subseteq\mathcal{P}(\X)$.
\begin{itemize}
 \item If $m\geq \kappa_\E^s <\infty$, then
$\operatorname{Mixt}^m(\mathcal{E}) = \mathcal{P}$.
 \item $\operatorname{Mixt}^m(\overline{\mathcal{E}})\supseteq\overline{\mathcal{ E}'}$ implies $m\geq \kappa_{\E,\E'}^f$.
\end{itemize}
\end{lemma}
\begin{Proof}
See Appendix.
\end{Proof}
\begin{remark}{\rm
If $\Mixt^m(\Ecal)=\Pcal$, then also
$\Mixt^m(\overline{\Ecal})=\overline{\Pcal}$, and if
$\Mixt^m(\overline{\Ecal})\neq\overline{\Pcal}$, then
$\overline{\Pcal}\setminus \Mixt^m(\Ecal)$ has a non-empty
interior. If
$\operatorname{Mixt}^m(\overline{\mathcal{E}})=\overline{\mathcal{P}}$,
then $m\geq\max \kappa_\E^f$, and if $\kappa_{\E,\E'}^f=\infty$,
then $\conv(\E) \not\supset \E'$.
Lemma~\ref{mixturedecompk} can be formulated for arbitrary models
as well. In that case however, the implication of the first item
holds only for the closures: If $m\geq \kappa_\Mcal^s$, then
$\operatorname{Mixt}^m(\overline{\mathcal{M}}) =
\overline{\mathcal{P}}$.}
\end{remark}

\begin{example}\label{proposuno}
Any distribution $p$ with support in a cylinder set
$[y_{\Lambda^c}]$, $\Lambda\subseteq [N]$, $|\Lambda|=k$ is
contained in the closure of the $k$-interaction family
$\overline{\E^k}$. Indeed, if $p\in\overline{\mathcal{P}}$ is
arbitrary with support $[y_{\Lambda^c}]$, then
$p(x)=\underset{\alpha\to\infty}{\lim} \exp(f(x_\Lambda) - \alpha
\sum_{j\in\Lambda^c} g_j(x_{j}))/{Z}$, where $Z$ is a
normalization constant, $f(x)=f(x_\Lambda)$ is a function of the
variables $X_i, i\in\Lambda$ with $f(x_\Lambda)=\log(p(x))+\log(Z)
\;\forall\, x\in[y_{\Lambda^c}]$, and $g_j$ is a function of
$X_j$, only, taking value $0$ for $x_j=y_j$ and $1$ otherwise.
Therefore, every $k$-dimensional cylinder set is an $S$-set of
$\E^k$. In particular, if $\X=\{1,\ldots,q\}^N$, then
$\kappa_{\E^k}^s \leq q^{N-k}$ and
$\Mixt^{q^{N-k}}(\Ecal^k)=\Pcal$.
\end{example}

\newpage
\begin{lemma}\label{sproperty}
Consider an exponential family $\E\subseteq\mathcal{P}(\X)$ and some $\Y\subseteq\X$. The following items are equivalent:
\begin{itemize}
 \item $\overline{\Ecal}\supseteq\overline{P}(\Y)$, i.\,e., $\Y$ is an {\em $S$-set}.
 \item $\conv\{A_y\}_{y\in\Y}$ is a $(|\Y|-1)$-dimensional simplex face of the convex support $Q$.
 \item $\supp(m^\pm)\not\subset \Y$ for all $m\in\ker(A)\subset\mymathbb{R}^\X\setminus\{0\}\,$, where $m^\pm(x):=\max\{0,\pm m(x)\}\;$ $\forall\, x\in\X$.
\end{itemize}
\end{lemma}
\begin{Proof}
The first item implies the second, because the moment map defines
a bijection between ${\mathcal{P}(\Y)}$ and
$\conv\{A_y\}_{y\in\Y}$. For the other direction: The matrix
$A_{\Y}:=(A_y)_{y\in\Y}$ defines an exponential family $\E_\Y
=\overline{\E}\cap\mathcal{P}(\Y)$, because $\Y$ is facial. If
$\conv \{A_y\}_{y\in\Y}$ is a $(|\Y|-1)$-simplex, then all columns
of $A_{\Y}$ are linearly independent ($\mathds{1}$ is a row of
$A$), and hence $\ker A_{\Y}=\{ 0\}$. As a consequence, any $p\in
\overline{\mathcal{P}}(\Y)$ trivially satisfies $\prod_x
(p(x))^{m^+(x)}-\prod_x (p(x))^{m^-(x)}=0 \;\forall\, m\in\ker
A_\Y$, which implies
$p\in\overline{\E}_\Y$~\cite{geiger:2006,RKA10:Support_Sets_and_Or_Mat}.
The~third item is equivalent to: $\Y$ is facial,
see~\cite{RKA10:Support_Sets_and_Or_Mat}, and additionally
$\supp(m) \not\subset\Y\;$ $\forall\, m\in\ker(A)\setminus\{0\}$.
This implies $\ker A_{\Y}= \{0\}$.
\end{Proof}

\begin{remark}{\rm
By Lemma~\ref{sproperty}, $\overline{\E}$ contains any $p$ with $|\supp (p)|< |\supp (m^+)|$ for all $m\in\ker(A)\setminus\{0\}$, and there always exists some $q\in{\overline{\Pcal}(\X)}\setminus\overline{\Ecal}$ with 
$$|\supp (q)|=\min_{m\in\ker(A)\setminus\{ 0 \}}|\supp (m^+)|.$$ 

When every column $A_x$ of the sufficient statistics is a vertex of $Q$ and $\kappa$ simplex faces of $Q$ contain all $A_x$, then  $\Mixt^\kappa(\overline{\Ecal})=\conv(\overline{\Ecal})$.
When all $A_x,x\in\X$ are distinct vertices of $Q$, then $\overline{\Ecal}$ contains all possible point measures,
$\kappa_\Ecal^s$ is the smallest number of simplex faces that contain all vertices, and $\Mixt^{\kappa_\Ecal^s}(\overline{\Ecal})=\overline{\Pcal}$.
Computing $\kappa_\Ecal^s$ can be difficult, in general.
Two examples of related problems are: Finding minimum clique coverings, which  is a graph-theoretical NP-complete problem, and describing perfect covering codes on $\{0,1\}^N$, which so far are not completely understood (see~\cite{CohenCoveringCodes}). \\

A polytope $P$ is called {\em $K$-neighborly}, when the convex
hull of any $K$, or less, of its vertices is a face
(see~\cite{Kalai:1993,Shemer:1982}). If the convex support of $\E$
is $K$-neighborly and $\overline{\Ecal}$ contains all point
measures, then every $\Y\subseteq\X$ with $|\Y|\leq K$ is an
$S$-set of $\Ecal$. It is known that the convex support $Q_k$ of
the $k$-interaction family is $(2^k-1)$-neighborly,
see~\cite{kahle:2008}. In Section~\ref{section:mixtureskmodel} we
will study the {\em simpliciality} of $Q_k$ and corresponding
vertex set coverings using simplex faces. A polytope $P$ is
$K$-simplicial if all its $K$-dimensional faces are simplices
(this does not mean that any $(K+1)$ vertices define a face of
$P$).}
\end{remark}

\begin{example}\label{polyex}
The convex support of the two-interaction family $\mathcal{E}^2$ on $\Xcal=\{0,1\}^4$ is a polytope with $16$ vertices and dimension $10$.
We computed the face lattice of $Q_2$ (using the software \texttt{Polymake}~\cite{polymake}).
We found $56$ facets (proper faces of maximal dimension, $9$), out of which $16$ are simplices.
One of them is $\conv\{A_x\}_{x\in\Y}$,
$ \Y=\{(0     0     0     0),
  (1     0     0     0),$
$  (0    1     0     0),
  (0     0    1     0),$
$  (1     0     0    1),
  (0    1     0    1),
  (0     0    1    1),$
$  (1    1     0    1),
  (1     0    1    1),
  (0    1    1    1)\}$.
In total $8$ $S$-sets of $\Ecal^2$ contain $6$ binary vectors with an even number of ones, and $8$ contain $6$ vectors with an odd number of ones.
The other $40$ facets have each $12$ vertices. Denote the $S$-sets
(of cardinality $10$) by $F_i, i=1,\ldots,16$ and the  facial sets
of cardinality $12$  by $G_i, i=1,\ldots,40$. We found that
$F_i\cup F_j\neq \X$ $\forall\, i,j$ and $F_i\cup G_j\neq \X$
$\forall\, i, j$. Since all faces  (and in particular all simplex
faces) are subsets of some facet, at least $3$ $S$-sets of
$\Ecal^2$ are needed to cover $\X$.
\end{example}

\begin{example}\label{ex2}
Let $\X=\{0,\ldots,n-1\}$ and let $\E$ be an exponential family with convex support an $n$-gon (a polygon with $n$ vertices).
We call this family an $n$-gon exponential family.
It is two-dimensional and contains all point measures $\delta_x$ in its closure. $n$-gon exponential families  have  been studied in the context of model design in~\cite{AMR2011}.
Assume that the boundary of the convex support $Q$ of an $n$-gon family  is the polyline $A_0A_{1}\cdots A_{n-1}A_0$.
The facial sets are: $\X$, the pairs $\{i,i+1\} \mod n$, and the points $\{i\}$ for ${i\in\X}$.
All facial sets, except $\X$, are $S$-sets.
The sample space $\X$ is covered by $\kappa^s_{\E}=\left\lceil\frac{n}{2}\right\rceil$ $S$-sets, while the packing of any set $\Y\subseteq\X$ requires at most $\max\kappa^f_{\E}=\left\lfloor\frac{n}{2}\right\rfloor$ facial sets.
By Lemma~\ref{mixturedecompk} the smallest $m$ for which $\Mixt^m(\E)=\operatorname{conv}(\E)=\mathcal{P}$ satisfies $\left\lfloor\frac{n}{2}\right\rfloor\leq m \leq \left\lceil\frac{n}{2}\right\rceil$.
For $n=5$ (see Figure~\ref{figurepentagon} right) we show that $m\geq 2=\left\lfloor\frac{n}{2}\right\rfloor$ is necessary and sufficient, see below.
\end{example}

\begin{figure}
\begin{center}
\setlength{\unitlength}{7cm}
\begin{picture}(.4,.75)(0.2,0)
\put(0,0){\includegraphics[width=4 cm]{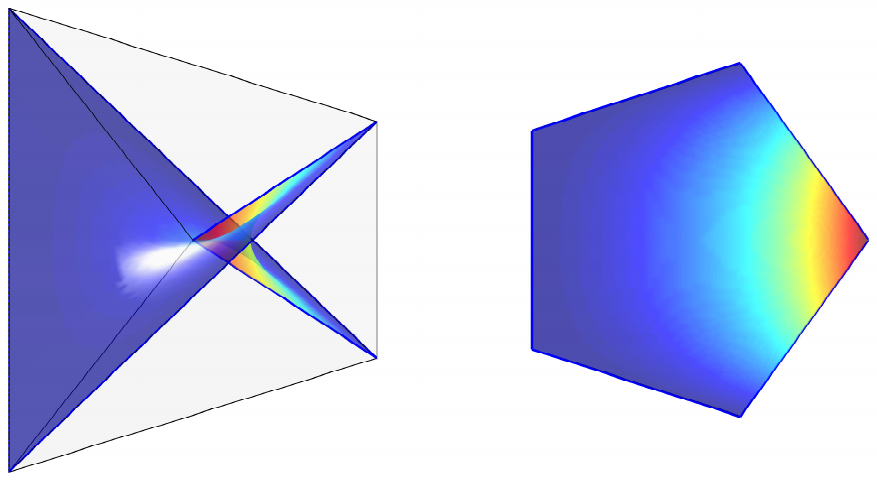}}
\put(.31,.22){$\mathcal{E}_{\mysymbolpentagon}$}
\put(-.02,.685){$\delta_{1}$} \put(.54,.55){$\delta_{3}$}
\put(-.02,.0){$\delta_{2}$} \put(.54,.148){$\delta_{0}$}
\put(.25,.38){$\delta_{4}$}
\end{picture}
\begin{picture}(0.4,.75)(0.4,0)
\put(0.62,0.16){\includegraphics[width=2.5
cm]{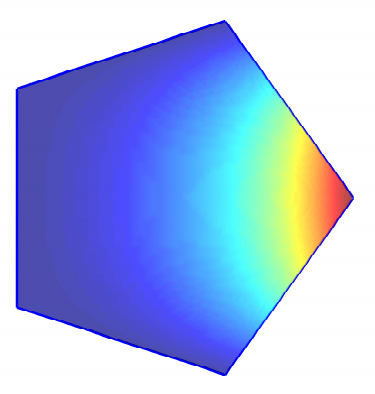}}
\put(.75,.345){$Q_{\mysymbolpentagon}$} \put(.6,.47){$A_{1}$}
\put(.8,.52){$A_{0}$} \put(.6,.19){$A_{2}$} \put(.8,.13){$A_{3}$}
\put(.95,.32){$A_{4}$}
\end{picture}
\end{center}
\caption{Schlegel diagram of the four-dimensional probability
simplex on $\{0,\ldots,4\}$ and corresponding projection of the
two-dimensional exponential family
$\mathcal{E}_{\mysymbolpentagon}$ with convex support
$Q_{\mysymbolpentagon}$ a regular pentagon. The color indicates
the value that the distributions take on $x=4$; blue for $p(4)=0$
and red for $p(4)=1$. The uniform distribution
$\frac{\mathds{1}}{5}$ and the point measure $\delta_4$ are
projected into the same point. }\label{figurepentagon}
\end{figure}

\begin{proposition}\label{pentag}
If $\E_{\mysymbolpentagon}$ is an exponential family on $\X=\{0,1,2,3,4\}$ with pentagonal convex support, then $\Mixt^2({\E}_{\mysymbolpentagon})={\mathcal{P}}(\X)$.
\end{proposition}
\begin{Proof} See Appendix. \end{Proof}

\begin{remark}{\rm
Example~\ref{ex2} shows that in general $\kappa_\Ecal^f\neq
\kappa_\Ecal^s$. In such a case $m=\kappa_\Ecal^s$ is not
necessarily the smallest $m$ for which $\Mixt^m(\Ecal)=\Pcal$. For
pentagonal exponential families $\kappa_\Ecal^s$ is off by one,
and the same likely happens for all $n$-gon exponential families
with odd $n$ greater or equal to five. However, in the next
section we show that $\kappa_\Ecal^f$ equals $\kappa_\Ecal^s$ for
many independence models, and we believe that this generalizes to
many interaction models.}
\end{remark}

\section{Mixtures of Hierarchical Models}\label{section:mixtureskmodel}

\subsection{Independence Models}
The  Hamming distance  between two vectors $x=(x_1,\ldots,x_N)$ and  $y=(y_1,\ldots,y_N)$ is $d_H(x,y):=|\{i\in[N]\colon x_i\neq y_i\}|$. A set $\Ycal\subseteq\X$ has minimum distance $d$ if the smallest Hamming distance between two distinct points $x,y\in\Y$ is at least $d$.
The  independence model of $N$ variables with joint sample space $\X=\times_{i\in[N]}\X_i$ is:
\begin{eqnarray}
{\E^1 }&=&\Big\{p\in\Pcal :p(x_1,\ldots,x_N)=\prod_{i\in[N]}
p_i(x_i) \;\forall\, (x_1,\ldots,x_N)\in\X , \\
&&\qquad \text{ and }p_i\in {\mathcal{P}}(\X_i) \;\forall\,
i\in[N]\Big\}\;.\nonumber
\end{eqnarray}
For binary variables the convex support of $\Ecal^1$ is a combinatorial $N$-cube, the facial sets are the cylinder sets (including those of dimension zero), and the $S$-sets are the pairs  of vectors with Hamming distance one to each other, plus all individual binary vectors.
In general, the convex support  is a Cartesian product $Q_1=\times_{i\in[N]} S_i$, where $S_i$ is a ($|\X_i|-1$)-dimensional simplex for every $i\in[N]$.
The facial sets are:
\begin{equation}
\Fcal(\Ecal^1) = \{\times_{i\in[N]} \Y_i: \Y_i\subseteq\X_i  \text{ for all } i\in[N]\}\;,
\end{equation}
and the $S$-sets are the subsets of one-dimensional cylinders;
i.\,e., the sets
\begin{equation}
\{y_1\}\times\cdots\times\{y_{i-1}\}\times\Y_i\times\{y_{i+1}\}\times\cdots\{y_N\}
\end{equation}
 with $y_j\in\X_j$ for all $ j\in [N]\setminus\{i\}$   and $\Y_i\subseteq\X_i $  for $ i\in[N]$.

We consider the maximal cardinality of subsets $\Y\subseteq\X=\times_{i\in[N]}\X_i$ with minimum distance  two:
\begin{equation}
\Acal_\X:=\max \{|\Y|:\Y\subseteq \X \text{ and } d_H(x,y)\geq
2\;\forall\, x,y\in \Y, x\neq y\}\;.
\end{equation}
For example, the sets of binary vectors of length $N$ with an even (odd) number of ones,
$Z_{\pm}:=\big\{x\in\X=\{0,1\}^N\colon \prod_{i\in [N]} (-1)^{x_i}=\pm 1\big\}$,
have the largest possible cardinality $|Z_\pm|=2^{N-1}$ among all sets of length-$N$ binary vectors of minimum distance two.
For mixtures of independence models we have:
\begin{theorem}\label{mixtqary}
The mixture model $\Mixt^m(\overline{\Ecal^1})$ contains every
probability distribution with support in a union of $m$
one-dimensional cylinder sets, and does not contain any
probability distribution supported by a set of cardinality more
than $m$ and minimum distance (at least) two. Furthermore,
\begin{itemize}
 \item
If $m\geq {|\X|}/\max\{|\X_i|\}_{i\in[N]}$, then $\operatorname{Mixt}^m(\E^1) = \mathcal{P}$.
 \item
If $\operatorname{Mixt}^m(\overline{\E^1}) \supseteq \mathcal{P}$,
then $m\geq \Acal_\X \geq \max\{\frac{s^{N}}{1+N(s-1)}, q^{N-1}
\}$,\newline where $s=\min\{|\X_i|\}_{i\in[N]}$ and $q$ is the
largest prime power smaller or equal to~$s$.
\end{itemize}
In particular, when $\X=\{1,\ldots,q\}^N$ and $q$ is a prime power, then
\[
\operatorname{Mixt}^m(\E^1) = \mathcal{P} \quad\text{ if and only if }\quad m\geq q^{N-1}\;.
\]
\end{theorem}
\begin{Proof}
For the first statement: Any face of $Q_1$ which has more than one vertex contains edges.
An edge of $Q_1$ is the convex hull of a column pair $A_x,A_y$ of the sufficient statistics, with $d_H(x,y)=1$.
Therefore, any facial set contained in a set that does not contain pairs of Hamming distance one, has cardinality one.
For the remaining statements, assume (without loss of generality)
$|\X_1|=\max\{|\X_i|\}_{i\in[N]}$ and
$|\X_N|=\min\{|\X_i|\}_{i\in[N]}$. The first bullet is by
Lemma~\ref{mixturedecompk}, covering the sample space with the
$S$-sets $\{(x_1,y_2,\ldots,y_N)\colon x_1\in\X_1\}$ for all
$(y_2,\ldots,y_N)\in\times_{i\in[N]\setminus\{1\}}\X_i$. For the
second bullet we use the first part of this theorem, and the fact
that $\overline{\Pcal}(\{1,\ldots,s\}^N)$ is contained in
$\overline{\Pcal}(\X)$. For any $q$ the {\em maximal cardinality
of a $q$-ary code of length $N$ and minimum distance $d$}, defined
as $\Acal_q(N,d):=\max \{|\Y|:\Y\subseteq \{1,\ldots,q\}^N \text{
and } d_H(x,y)\geq d\;\forall\, x,y\in \Y, x\neq y\}$,
is familiar in coding theory. It is known that
$\Acal_q(N,2)\geq\frac{q^N}{\sum_{j=0}^{d-1}{N\choose j} (q-1)^j}$
({\em Gilbert-Varshamov
bound}~\cite{Gilbert:1952,Varshamov:1957}), and that when $q$ is
any power of a  prime number, $\Acal_q(N,d)\!\geq\! q^k$, where
$k$ is the largest integer with
$q^k\!<\!\frac{q^N}{\sum_{j=0}^{d-2}{{N-1}\choose j} (q-1)^j}$.
Evaluating these bounds for $d=2$ completes the proof.
\end{Proof}

\begin{corollary}\label{lemmaopti}\label{coroptimixturedecomp1}
Let $1\leq k\leq N-1$. If $\Xcal=\{0,1\}^N$ and $\operatorname{Mixt}^m(\overline{\E}) \supseteq \E^k$, then
\[
m  \geq \max\{|\Z| : \Z \in \Fcal(\E^k),\, \Z\subseteq Z_\pm\} \geq 2^k -1\;.
\]
\end{corollary}
\begin{Proof}
The first inequality is by Lemma~\ref{mixturedecompk},
since $\kappa_{1,k}^f  \geq \max\{|\Z| : \Z \in \Fcal(\E^k)$. The second one follows from Lemma~\ref{cyclic}.
\end{Proof}

\begin{example}\label{exese} \mbox{}
The first inequality in Corollary~\ref{lemmaopti} is useful when we have information about the support sets of $\overline{\Ecal^k}$.
The second inequality improves the bound $$m\geq (\dim(\Ecal^k)+1)/(N+1)=\sum_{j=0}^k{N\choose j}/{(N+1)}$$ that can be  derived comparing the dimension of both models, when $k$ is close to $N$.
For instance:
\begin{itemize}
\item
When $\X=\{0,1\}^4$, $\E^2$ has $S$-sets of cardinality $6$, contained in $Z_+$ (see Example~\ref{polyex}).
Hence, if $\operatorname{Mixt}^m(\overline{\E^1}) \supseteq\E^2$, then $m\geq 6$.
\item
If $\Xcal=\{0,1\}^4$ and $\operatorname{Mixt}^m(\overline{\E^1}) \supseteq\E^3$, then $m\geq 7$. For comparison, counting parameters yields only $m\geq \left\lceil(\dim(\E^3)+1)/(4+1)\right\rceil = 3$.
\end{itemize}
\end{example}

\newpage
\subsection{Interaction Models}

\begin{theorem}\label{decompfromcubeplus}
Consider a hierarchical model $\Ecal_\Delta$ on $\Xcal=\times_{i\in[N]}\X_i$ with $\Delta\supseteq\Delta_k$,\linebreak $1\leq k<N$.
\begin{itemize}
\item The mixture model
$\operatorname{Mixt}^m(\overline{\mathcal{E}_\Delta})$ contains
any probability distribution $p\in\overline{\Pcal}(\X)$ \newline when $m$
is larger or equal to $\min\{n\colon \exists\, \Y_1,\ldots,\Y_n
\text{ $k$-cylinder sets of $\X$ with }\linebreak
\supp(p)\subseteq \cup_{i=1}^n \Y_i \}$. Furthermore,
$\Mixt^m(\Ecal_\Delta)=\Pcal$ whenever \\ $m\geq
|\X|/\max\{\prod_{i\in\lambda}|\X_i| \colon
\lambda\subseteq[N], |\lambda|=k\}$. 
\item In the case of binary
variables, the convex support $Q_\Delta$ is $(2^k\!-1)$-neighborly,
$(2^{k+1} -3)$-simplicial, and all its vertices are contained in
the union of $2^{N-(k+1)}(1+\frac{1}{2^k-1})$ simplex faces. Moreover,
$\operatorname{Mixt}^m(\mathcal{E}_\Delta)=\mathcal{P}$ whenever
$m\geq 2^{N-(k+1)}(1+\frac{1}{2^k-1})$.
\end{itemize}
\end{theorem}

The first item of Theorem~\ref{decompfromcubeplus} follows from the observation that all $k$-cylinders are $S$-sets of $\Ecal^k$, see Example~\ref{proposuno}.
The $(2^k-1)$-neighborliness of $Q_\Delta$ was shown in~\cite{kahle:2008}.
The $(2^{k+1}-3)$-simpliciality follows from a classic result of convex polytopes, which  states that  if $P$ a $K$-neighborly $d$-dimensional polytope, then every face $F$ of dimension less than $2K$ is a simplex,
see~\cite[Theorem~7.4.3]{Gruenbaum:2003}.
In order to prove the remaining statements of the theorem, we need to find the $(2^{k+1}-3)$-dimensional faces of $Q_\Delta$ and show that at most $2^{N-(k+1)}(1+\frac{1}{2^k-1})$ of them cover all vertices.
Before proving this, some remarks are appropriate:

\begin{remark}{\rm
Regarding the upper bound  $2^{N-(k+1)}(1+\frac{1}{2^k-1})$ on the
minimal cardinality of an $S$-set covering of $\{0,1\}^N$ (second
item of Theorem~\ref{decompfromcubeplus}): When  $k=1$ the bound
equals $2^{N-1}$ and is tight by Theorem~\ref{mixtqary}. When
$N=4$ and $k=2$ the bound is
$\left\lceil2^{4-(2+1)}/(1-2^{-2})\right\rceil=3$ and is tight in
view of Example~\ref{polyex}. When $k=N-1$ the bound equals $2$
and is tight, because
$\Mixt^1(\Ecal_\Delta)=\Ecal_\Delta\neq\Pcal$. In spite of this,
the characterization of the simplex faces of convex support
polytopes and the computation of the smallest
simplex-face-vertex-set coverings
 for hierarchical models with general interaction sets $\Delta$ and non-binary variables, is not fully accomplished at this point.
In particular we believe that the bound provided in the first item
can be further improved, as for binary variables the second item
provides a much better bound.}
\end{remark}

Note that any facial set of $\Ecal^k$ is a facial set of $\Ecal_\Delta$, $\Delta\supseteq\Delta_k$.
For $0 <k<N$, any $(k+1)$-dimensional cylinder set
$[y_{\lambda^c}], \lambda\subseteq[N], |\lambda|=k+1$
is a facial set of $\E^k$ (for example by similar arguments as in Example~\ref{proposuno}). Hence  the vertices of $Q_\Delta$ can be covered by $2^{N-(k+1)}$ disjoint faces $\{F_i\}_i$ corresponding to $(k+1)$-dimensional cylinder sets.
These $F_i$ are not simplices, but they contain $(2^{k+1}-3)$-dimensional simplex faces (see below), which we can arrange in a convenient way to cover all vertices of $Q_k$ disjointly.
We use the following Lemma~\ref{cyclic}, which subsumes various ideas and remarks from~\cite{Gruenbaum:2003,Hosten2002,kahle:2009}.

A $d$-dimensional {\em cyclic polytope} with $v$ vertices (see~\cite{Gale:1963}) is defined as the convex hull of $v$ distinct points on the $d$-{\em moment curve}:
$ C(v,d):=\conv\{x(t_i)\}_{i=1,\ldots,v}$, where $v\geq d+1$, $t_1<\cdots<t_v\in\R$, and $x(t) = (t,t^2,\ldots,t^d)\in\R^d$.

\begin{lemma}\label{cyclic}
Let $0< k<N$ and $\X=\{0,1\}^N$. Any $(k+1)$-dimensional cylinder
set $\Y$ is a facial set of $\mathcal{E}^k$ and the corresponding
face $F$ of the convex support $Q_k$ is a simplicial polytope,
combinatorially equivalent to the cyclic polytope
\mbox{$C(2^{k+1},2^{k+1}\!-\!2)$}. There are exactly ${2^{2k}}$
$S$-sets of cardinality $(2^{k+1}-2)$ contained in $\Y$;
namely $\{\Z\subset \Y:  \Y\cap Z_\pm\not\subseteq \Z  \}$. In
particular, if a set $Z\subseteq\X$ contains $\Y\cap Z_\pm$ but
does not contain $\Y$, then $\Z$ is not facial.
\end{lemma}
\begin{Proof} See Appendix. \end{Proof}
\bigskip

\noindent P\,r\,o\,o\,f\,  o\,f  \, T\,h\,e\,o\,r\,e\,m\, \ref{decompfromcubeplus}. 
Let $x_i^{i+k}:=(x_i,\ldots,x_{i+k})\in \{0,1\}^{\{i,\ldots,i+k\}}$. Consider the following partition of $\{0,1\}^N$ into $(k+1)$-dimensional cylinder sets:
\[
C_y :=\{(x_1^{k+1},x_{k+2}^N)\in\{0,1\}^N \colon x_{k+2}^N=y\} \quad\text{for all}\quad y\in\{0,1\}^{N-(k+1)}\;.
\]
By Lemma~\ref{cyclic} the elements of any  $C_y$ can be disjointly covered by:

\noindent{\em (i)}  An $S$-set of $\overline{\mathcal{E}^k}$ of cardinality $2^{k+1}-2$. We denote this set by $G_y$.\\
\noindent{\em (ii)} A pair of vectors differing in one entry:
\begin{equation}
E_y:=\{(z_1^{k},x_{k+1},y)\in\{0,1\}^N\colon z_1^{k}   \text{
fixed}\}\;.  \label{pairobs}
\end{equation}

The vector $z$ in eq.~\eqref{pairobs} can be chosen equal for all $E_y$, such that the $S$-sets $\{G_y\}_y$ satisfy:
\[
\bigcup_{y\in\{0,1\}^{N-(k+1)}} G_y \quad =\quad \{0,1\}^N\setminus \tilde C_{N-k}\;,
\]
where $\tilde C_{N-k}$ is the following $(N-k)$-dimensional cylinder set:
\[
\tilde C_{N-k} = \bigcup_{y\in\{0,1\}^{N-(k+1)}} E_y \quad=\quad \{(z_1^k,\tilde y_1^{N-k}): z_1^k \text{ fixed}\} \;.
\]
The cylinder set $\tilde C_{N-k}$ can be considered as a new sample space which still has to be covered using as few $S$-sets as possible.
If $N-k < k+1$, only one $S$-set is required.
Iteration of the previous idea until exhausting all coordinates yields that $\kappa$, the minimal number of simplex faces of $Q_k$ covering all vertices, is not more than: 
\[
\kappa\leq 1+ \sum_{0\leq i\leq
\frac{N-(k+1)}{k}}\frac{2^{N-ik}}{2^{k+1}} =
\left\lceil\frac{2^N}{2^{k+1}}\sum_{i=0}^\infty \frac{1}{(2^k)^i}
\right\rceil = \left\lceil
\frac{2^{N-(k+1)}}{1-2^{-k}}\right\rceil\;.
\]
\makebox{}\hfill$\Box$

\bigskip

We conclude this section with a few observations on $S$-sets of hierarchical models.
From Lemma~\ref{cyclic} we can derive a rough cardinality upper bound for the $S$-sets of $\Ecal^k$.
Let $K(N,k+1)$ denote the smallest cardinality of a set $\Y\subseteq\{0,1\}^N$ which intersects all $(k+1)$-dimensional cylinder sets, and let $B_{N,k+1}$ denote a Hamming ball in $\{0,1\}^N$ of radius $k+1$.
\begin{proposition}\label{coroese}
 If $\Y\subseteq\X=\{0,1\}^N$ is an $S$-set of ${\mathcal{E}^k}$, then
 $$|\Y\cap Z_\pm|\leq 2^{N-1}-K(N,k+1)\leq 2^{N-1}(1-2/|B_{N,k+1}|)$$ and $|\Y|\leq |\Delta_k|$.
Furthermore, $$|\Y|\leq 2^N - 2K(N,k+1)\leq 2^N (1-2/|B_{N,k+1}|),$$ since $\X$ is disjointly covered by the two sets $Z_+$ and $Z_-$.
\end{proposition}
\begin{Proof} See Appendix. \end{Proof}

\begin{example}
When $\X=\{0,1\}^4$, by Proposition~\ref{coroese} any $S$-set  of $\Ecal^2$  intersects $Z_+$, or $Z_-$, at most at $8-2=6$ points.
This bound is attained exactly, in view of Example~\ref{polyex}.
\end{example}

It is worthwhile mentioning that, if a collection of index sets $\Delta\subseteq2^{[N]}$ is symmetric with respect to a permutation $\pi:[N]\to [N]$, then the convex support $Q_\Delta$ of the associated exponential family, also has this symmetry. In particular, if $\Y$ is an $S$-set of $\mathcal{E}^k$, then $\pi(\Y):=\{(x_{\pi(1)},\ldots,x_{\pi(N)}):x\in\Y\}$ is also an $S$-set, for any permutation $\pi$.
Furthermore, we have:
\begin{proposition}\label{unarelacionsobredist}
If $\mathcal{E}$ is an exponential family with sufficient
statistics
$$A=((-1)^{|\supp(x)\cap\lambda|})_{\lambda\in\Delta,x\in\X},\ \Delta\subseteq2^{[N]},\ \Xcal=\{0,1\}^N,$$ then $\Y$ is an
$S$-set if and only if  $x*\Y:=\{x+y \mod 2: y\in\Y \}$ is an
$S$-set $\forall\, x\in\X$, and moreover,   $\Y\subseteq\X$ is a
facial set if and only if $x*\Y$ is a facial set $\forall\,
x\in\X$.
\end{proposition}
\begin{Proof}See Appendix. \end{Proof}

\bigskip

The sets $\{x*\Y\}_{x\in\X}$ are not necessarily all different
from each other, but they are if $|\Y|$ is odd, or if $\Y$ is a
Hamming ball. The orbit $\{x*z:x\in\X\}$ of any $z\in\X$, covers
$\X$. In particular, $\cup_{x\in\X} x*\Y=\X$ and $|x*\Y|=|\Y|$ for
any $x\in\X$, $\Y\subseteq\X$, $\Y\neq\emptyset$. These
observations have interesting relations to coding theory; for
example,  any binary hierarchical model has a convex support
which is the convex hull of a binary linear code,
see~\cite{kahle:2009}.

\section*{Appendix}

\noindent P\,r\,o\,o\,f\,  o\,f  \, L\,e\,m\,m\,a\, \ref{mixturedecompk}.\\[1mm]
1. Let $\{\Y_i\}_{i=1}^m$ be an $S$-set covering of $\X$. W.l.o.g.
$\Y_i\cap\Y_j=\emptyset\;\forall\, i\neq j$. Any
$p\in\overline{\mathcal{P}}$ can be written as $\sum_{i=1}^m
\alpha_i f_i$ and $f_i\in\overline{\mathcal{E}}$ choosing $f_i$
with $\supp(f_i)\subseteq\Y_i$,
$f_i=p|_{\Y_i}/\sum_{x\in\Y_i}p(x)$ and
$\alpha_i=\sum_{x\in\Y_i}p(x)$. This shows
$\operatorname{Mixt}^m(\overline{\mathcal{E}})=\overline{\mathcal{P}}$.

For strictly positive distributions: The convexity of $\Pcal$
implies $\operatorname{Mixt}^m(\mathcal{E})\subseteq\mathcal{P}$
for all $m\geq1$. The direction ``$\supseteq$'' is a bit more
elaborate. By the first part of this proof  the set
$\Mixt^m(\Ecal)$ is dense in $\overline{\Pcal}$; we need to show
that within $\overline{\Pcal}$ only the boundary
$\partial\Pcal:=\overline{\Pcal}\setminus\Pcal$ is  not contained
in $\Mixt^m(\Ecal)$. We use topological arguments. Let
$Y_i:=\overline{\mathcal{P}}(\Y_i)$, $i=1,\ldots,m$ be disjoint
faces of $\overline{\Pcal}$ containing all point measures
$\{\delta_x\}_{x\in\X}$.
Let $p_\eta:=(A|_{\overline{\E}})^{-1}(\eta)$ denote the distribution in $\overline{\Ecal}$ with expectation parameter $\eta$.
The mixture map $\phi: D:=\overline{\mathcal{P}}_m\times(\times_{i=1}^m Q) \to \overline{\mathcal{P}}\;;\; (\alpha,\eta_1,\ldots,\eta_m)\mapsto\sum_{i=1}^m \alpha(i) p_{\eta_i}$ is surjective.
Restricting the domain $D$ to the subset $C:=\partial (\mathcal{P}_m\times(\times_{i=1}^m (A\cdot Y_i)))$ we get a  continuous bijection
$\phi|_C\colon C\to\partial{\mathcal{P}}$ between the compact domain $C$ and the Hausdorff codomain $\partial{\mathcal{P}}$. Therefore, $\phi|_C$ is a homeomorphism and induces isomorphisms between the homotopy groups of $C$ and those of $\partial\mathcal{P}\simeq S^{|\X|-2}$, the $(|\X|-2)$-sphere.
Denote by $\mathring{D}$ the relative interior of the polytope $D$. Note that $\phi(\mathring{D})\subseteq\mathcal{P}$.
For any $\epsilon>0$ there is a continuous deformation $C\to \tilde C\subseteq \mathring{D}$ which is mapped by $\phi$ into a continuous deformation $\partial \mathcal{P} \to \phi(\tilde C) \subset \mathcal{P}\setminus\mathcal{P}^\epsilon$, $\mathcal{P}^\epsilon:=\{p\in\mathcal{P}\colon p(x)\geq \epsilon\;\forall\, x\in\X\}$. If $\phi(\mathring{D})$ does not contain $\mathcal{P}^\epsilon$, then $\phi(\tilde C)$ is not contractible in $\phi(\mathring{D})$, in contradiction to the contractibility of $\mathring{D}$ (which is a convex set). Since any element of $\mathcal{P}$ belongs to some $\mathcal{P}^\epsilon$, this shows $\Mixt^m(\E)\supseteq\mathcal{P}$.\\
\indent
2. Consider some $p\in\overline{\E'}$ with $\supp(p)=\Z\in \F(\E')$.
If $p$ is written as a mixture of elements from $\overline{\E}$, then every mixture component with positive mixture weight must have a support $\Y\in\F(\E)$, $\Y\subseteq\Z$. Furthermore, the union of the support sets of these summands must equal $\Z$. The minimal number of summands is, by definition, equal to $\kappa_\E^f(\Z)$.
\hfill$\Box$
\\

\noindent P\,r\,o\,o\,f\,  o\,f  \, P\,r\,o\,p\,o\,s\,i\,t\,i\,o\,n\, \ref{pentag}.
Consider any exponential family $\Ecal$, and assume (without loss of generality) that the sufficient statistics contains the row $\mathds{1}$.
The image of the moment map $\pi\colon p\mapsto A\cdot p$ is the convex support $Q=\conv \{A_x\}_{x\in\X}$.
Since $\pi$ is continuous, $\overline{\E}$ is compact, and $Q$ is Hausdorff, this bijective map is in fact a homeomorphism.
We denote by $p_\eta=(\pi|_{\overline{\E}})^{-1}(\eta)$ the unique preimage of $\eta\in Q$ by the moment map restricted to $\overline{\Ecal}$.
The $m$-mixture of $\oE$ is parametrized by a {\em mixture map} $\phi\colon D:=\oP_m\times Q^m \to \overline{\Pcal}\;;\; (\alpha,\eta_1,\ldots,\eta_m)\mapsto\sum_{i=1}^{m} \alpha_i p_{\eta_i}$.
Consider the {\em normal space}  $\N=\ker A$ of $\E$.
For any $p\in\oP$ the set $\N_p:=\{q\in\oP: p-q\in\N\}$ is a polytope of dimension $\dim \ker A$ which intersects $\oE$ at a unique point $p_\E\in\oE\cap\N_p$ (see~\cite[Theorem~2.16]{Rauh11:Thesis}).
Hence $\oP=(\oE+\ker A)\cap\oP= {\dotcup}_{ p\in\oE}\N_{p}$.
The boundary of $\N_p$ is contained in the boundary of $\Pcal$.

In the case of $\Ecal_{\mysymbolpentagon}$ $\dim\ker A = 2$.
Furthermore, any subset of $\X=\{0,1,2,3,4\}$ of cardinality $4$
is contained in the union of two $S$-sets. Hence
$\Mixt^2(\oE)\supset\partial\Pcal:=\oP\setminus {\Pcal}$, and the
restriction $\phi|_C\colon C:=\partial (\oP_2\times Q^2)
\to\partial{\Pcal}$ is a continuous surjection. Now, for any
$p\in\E_{\mysymbolpentagon}$ we consider the set $B_p =
\phi^{-1}(\N_p) = \{(\alpha,\eta_1,\eta_2)\in D
\colon\sum_{i=1}^2\alpha_i \eta_i=\pi (p)\}$. This set is mapped
by $\phi$ into the set of convex combinations of $2$ elements of
$\oE_{\mysymbolpentagon}$ which have the same expectation
parameter as $p$. We consider also $\partial B_p = B_p\cap
(\oP_2\times (\partial Q)^2)$, which corresponds to the same kind
of mixtures, but with mixture components from the boundary
$\partial\E_{\mysymbolpentagon}:=\overline{\E_{\mysymbolpentagon}}\setminus\E_{\mysymbolpentagon}$.
We have that $\phi\colon \partial B_p\to\partial\N_p$ is
surjective and has degree $2!$ (the cardinality of the preimage of
a regular value, which arises from the freedom to permute the
mixture components). The set $\partial B_p$ is parametrized by an
angle, say $\gamma$, and  $\phi|_{\partial B_p}(\gamma)$
circulates $\partial\N_p$ twice. Using that $B_p$ is contractible,
it follows that $\phi|_{B_p}=\N_p$ and
$\Mixt^2(\overline{\E_{\mysymbolpentagon}})=\oP$. For strictly
positive distributions the claim follows from the fact that
$\phi(\mathring{B}_p)\subseteq\Pcal$, and that the image of  an
$\varepsilon$-retraction of $B_p$, $(1-\varepsilon)(B_p -p)+p$,
can be made such that it contains any $\delta$-retraction of
$\N_p$, $(1-\delta)(\N_p-p)+p$.
\hfill $\Box$
\\

\noindent P\,r\,o\,o\,f\,  o\,f  \, L\,e\,m\,m\,a\, \ref{cyclic}.
By Lemma~\ref{sproperty} $\Y$ is not an $S$-set $\Leftrightarrow$ $\exists m\in\ker A\setminus\{0\}$ with $\supp(m^+)\subseteq\Y$.
If $\supp(m^+)=\Y$, then $\Y$ is not facial, see~\cite{geiger:2006,RKA10:Support_Sets_and_Or_Mat}.
Consider the sufficient statistics
$A\!=\!(\sigma_{\lambda,x})_{\lambda\in\Delta_k,x\in\X}$. The
kernel of this matrix is spanned by the rows of the matrix
$(\sigma_{\lambda,x})_{\lambda\in
2^{[N]}\setminus\Delta_k,x\in\X}$, which can be written as $
(\sigma_{\lambda,x})_{\lambda\in
\Delta_{N-k},x\in\X}\operatorname{diag}(\sigma_{[N],x})_{x\in\X}$.
The row span of $(\sigma_{\lambda,x})_{\lambda\in
\Delta_{N-k},x\in\X}$ contains any function of $(N-k)$ variables,
including the indicator function $\mathds{1}_\Y(x)$ of any
$(k+1)$-cylinder set $\Y$. This corresponds to a kernel element of
$A$ with entries $m(x)=\mathds{1}_\Y(x) \sigma_{[N],x}, x\in\X$
and $\supp(m^+)=Z_+\cap \Y$.

Since not all subsets of $\Y$ are facial, the corresponding face $F$ of the convex support, which has $2^{k+1}$ vertices,  is not a simplex and has dimension less than $2^{k+1}-1$.
By~\cite[Theorem~7.4.3]{Gruenbaum:2003} and the $(2^{k}-1)$-neighborliness of $Q_k$~\cite{kahle:2008}, $F$ is $(2^{k+1}-3)$-simplicial and has dimension less than $2^{k+1}-2$ (otherwise it would be a simplex).
The combinatorial equivalence of $F$ and the cyclic polytope $C(2^{k+1},2^{k+1}-2)$ follows from the fact that
{\em any $2n$-dimensional, $n$-neighborly polytope with $v\leq 2n+3$ vertices is combinatorially equivalent to the cyclic polytope $C(v,2n)$}~\cite[Theorem~7.2.3]{Gruenbaum:2003}.

To complete the proof we use Gale's Evenness Criterion: {\em A
$d$-tuple $V_J\!=\!\{x(t_j)\}_{j\in J}$ $J \subset [v], |J|=d$ of
vertices of $C(v,d)$, spans a facet iff between any two elements
of $J$ there is an even number of elements in $[v]\setminus
J$}~\cite[Theorem~4.7.2]{Gruenbaum:2003}.
In our case $v=2^{k+1}$ and $d=2^{k+1}-2$.
The combinatorial structure of the cyclic polytope is independent of the map $i\mapsto t_i$ and we may choose $I=[v]:=\{1,\ldots,2^{k+1}\}\subset \mymathbb{N}$. The sets $V_J, |J|=2^{k+1}-2$ satisfying the evenness criterion are exactly the complements of pairs $\{i^e, i^o\}\subset [v]$, where $i^e$ is even and $i^o$ is odd. There are ${2^{2k}}$ such pairs, and hence facets.
This is the same as the number of sets respecting the condition on $S$-sets, $\Z\not\supseteq \Y\cap Z_\pm$, shown at the beginning of this proof.
Therefore, all sets $\Z$ with  $\Z\not\supseteq \Y\cap Z_\pm$ correspond to facets of $C(2^{k+1},2^{k+1}-2)$ and are indeed $S$-sets.
\hfill $\Box$
\\

\noindent P\,r\,o\,o\,f\,  o\,f  \, P\,r\,o\,p\,o\,s\,i\,t\,i\,o\,n\, \ref{coroese}.
Let  $\Y$ be any $S$-set of $\mathcal{E}^k$ and let $C$ be any $(k+1)$-dimensional cylinder set.
By Lemma~\ref{cyclic} $|(C\cap Z_\pm) \setminus \Y|\geq 1$.
Therefore, the maximal cardinality of an $S$-set $\Y\subseteq Z_\pm$ is upper bounded by $|Z_\pm|-K(N,k+1)$, where $K(N,k+1)$ is the smallest cardinality of a set that intersects every $(k+1)$-dimensional cylinder set.
The union of all $(k+1)$-cylinder sets that contain a point $x$ equals the Hamming ball $B_{N,k+1}(x)\subseteq\X$ of radius $k+1$ centered at $x$.
Hence $K(N,k+1)$ is the minimal cardinality of a binary code of covering radius $k+1$.
If $R< N\leq 2R +1$, then $K(N,R)=2$, but in general computing $K(N,R)$ is hard (see~\cite{CohenCoveringCodes}).
A crude lower bound is the {\em sphere-covering bound}: $K(N,R)\geq 2^N / |B_{N,R}|$.
Here $|B_{N,R}|=\sum_{i=0}^R{N\choose i}$.
On the other hand, the cardinality of an $S$-set of $\mathcal{E}^k$ can not exceed $\dim Q_k +1=|\Delta_k|=|B_{N,k}|$, by parameter counting arguments.
\hfill $\Box$
\\

\noindent P\,r\,o\,o\,f\,  o\,f  \, P\,r\,o\,p\,o\,s\,i\,t\,i\,o\,n\, \ref{unarelacionsobredist}.
Consider the sufficient statistics
$A=(\sigma_{\lambda,x})_{\lambda\in\Delta,x\in\Xcal}$. We
abbreviate $(\sigma_{\lambda,x})_{\lambda\in\Delta,x\in\Y}$ by
$\sigma(\Delta, \Y)$. A set $\Y\subseteq\X$ is an $S$-set of
$\Ecal$ if and only if \noindent{\em (i)} $\operatorname{rk}
\sigma(\Delta,\Y)=|\Y|$, (i.\,e., $\Y$ describes a
$(|\Y|-1)$-simplex), and \noindent{\em (ii)} there exists a vector
$c\in\mymathbb{R}^{|\Delta|}$ for which $\langle c,
\sigma(\Delta,y)\rangle=0\;\forall\, y\in\Y$ and $\langle c,
\sigma(\Delta,x) \rangle\geq 1 \;\forall\, x\in \X\setminus\Y$,
(i.\,e., $\Y$ is a facial set). We show that $\Y$ satisfies these
properties if and only if $x*\Y$ does. We have that
\begin{eqnarray*}
\sigma({\lambda,x*y})&=&(-1)^{|(\supp(x){\text{\tiny{$\triangle$}}}\supp(y))\cap\lambda|}
=(-1)^{|\supp(x)\cap\lambda|}(-1)^{|\supp(y)\cap\lambda|}\\
&&\qquad \forall\, x\in\X ,\  \lambda\in 2^{[N]} ,\  y\in\X
\end{eqnarray*}
and thus
$\sigma(\Delta,x*\Y)=\operatorname{diag}\left(\sigma(\Delta,x)
\right)\cdot \sigma(\Delta,\Y)$. Hence $\operatorname{rk}
\sigma(\Delta,\Y)=\operatorname{rk}\sigma(\Delta,x*\Y)$. Consider,
on the other hand, the vector $\tilde
c:=\operatorname{diag}(\sigma(\Delta,x))\cdot c$. We have $\langle
\tilde c,\sigma(\Delta,x*y) \rangle=\langle c,\sigma(\Delta,y)
\rangle=0\;\forall\, y\in\Y$, and  $\langle\tilde
c,\sigma(\Delta,z')\rangle\geq 1$ $\forall\, z'\in \X\setminus
y*\Y$.
\hfill $\Box$
\\

\newpage
\bigskip
\section*{Acknowledgment}
\small I am grateful to Johannes Rauh, Thomas Kahle, and Nihat Ay
for many valuable discussions and comments. Furthermore, I am
grateful to Shun-ichi Amari for valuable discussions. I~thank
Jason Morton for help in proof-reading the manuscript. I~am
grateful to anonymous reviewers for very helpful suggestions. This
work was carried out  mostly at MPI-MIS, Leipzig, Germany; partly
at RIKEN BSI, Hirosawa, Saitama, Japan; and partly at PennState,
supported by DARPA grant FA8650-11-1-7145.

\makesubmdate


\begin{thebibliography}{10}

\bibitem{Amari:1999}
S.~Amari:
\newblock{Information geometry on hierarchical decomposition of stochastic
  interactions.}
\newblock{IEEE Trans. Inform. Theory {\mi 47} (1999), 1701--1711.}

\bibitem{amari:00a}
S.~Amari and H.~Nagaoka:
\newblock{Methods of information geometry, Vol. {\mi 191}.}
\newblock{Oxford University Press, 2000.}
\newblock{Translations of mathematical monographs.}

\bibitem{AyKnauf06:Maximizing_Multiinformation}
N.~Ay and A.~Knauf:
\newblock{Maximizing multi-information.}
\newblock{Kybernetika {\mi 42} (2006), 517--538.}

\bibitem{AMR2011}
N.~Ay, G.~F. Mont\'ufar, and J.~Rauh:
\newblock{Selection criteria for neuromanifolds of stochastic dynamics.}
\newblock{In: Advances in Cognitive Neurodynamics (III). Springer, 2011.}

\bibitem{Bishop:2006}
C.\,M. Bishop:
\newblock{Pattern Recognition and Machine Learning (Information Science
  and Statistics).}
\newblock{Springer-Verlag, New York 2006.}

\bibitem{BC2011}
C.~Bocci and L.~Chiantini:
\newblock{On the identifiability of binary segre products.}
\newblock{J. Algebraic Geom. {\mi 5} (2011).}

\bibitem{Brown86:Fundamentals_of_Exponential_Families}
L.~Brown:
\newblock{Fundamentals of Statistical Exponential Families: With
  Applications in Statistical Decision Theory.}
\newblock{Institute of Mathematical Statistics, Hayworth 1986.}

\bibitem{Catalisano2011}
M.\,V. Catalisano, A.\,V. Geramita, and A.~Gimigliano:
\newblock{Secant varieties of $\P^1\times\dots\times\P^1$
  ($n$-times) are not defective for $n\geq5$.}
\newblock{J. Algebraic Geom. {\mi 20} (2011), 295--327.}

\bibitem{diaconis:1977}
P.~Diaconis:
\newblock{Finite forms of de Finetti's theorem on exchangeability.}
\newblock{Synthese {\mi 36} (1977), 271--281.}

\bibitem{Efron:1978}
B.~Efron:
\newblock{The geometry of exponential families.}
\newblock{Ann. Statist. {\mi 6} (1978), 2, 362--376.}

\bibitem{CohenCoveringCodes}
S.\,L.\,G.~Cohen, I.~Honkala, and A.~Lobstein:
\newblock{Covering Codes.}
\newblock{Elsevier, 1997.}

\bibitem{Gale:1963}
D.~Gale:
\newblock{Neighborly and cyclic polytopes.}
\newblock{In: Convexity: Proc. Seventh Symposium in Pure
  Mathematics of the American Mathematical Society 1961,
  pp.\,225--233.}

\bibitem{polymake}
E.~Gawrilow and M.~Joswig:
\newblock{Polymake: a framework for analyzing convex polytopes.}
\newblock{In: Polytopes -- Combinatorics and Computation (G.~Kalai and G.\,M. Ziegler, eds.),
Birkh\"auser 2000, pp.\,43--74.}

\bibitem{geiger:2006}
D.~Geiger, C.~Meek, and B.~Sturmfels:
\newblock{On the toric algebra of graphical models.}
\newblock{Ann. Statist. {\mi 34} (2006), 1463--1492.}

\bibitem{Gilbert:1952}
E.~Gilbert:
\newblock{A comparison of signalling alphabets.}
\newblock{Bell System Techn. J. {\mi 31} (1052), 504--522.}

\bibitem{Gilula1979}
Z.~Gilula:
\newblock{Singular value decomposition of probability matrices: Probabilistic
  aspects of latent dichotomous variables.}
\newblock{Biometrika {\mi 66} (1979), 2, 339--344.}

\bibitem{Gruenbaum:2003}
B.~Gr{\"u}nbaum:
\newblock{Convex Polytopes. Second edition.}
\newblock{Springer-Verlag, New York 2003.}

\bibitem{henk:1997}
M.~Henk, J.~Richter-Gebert, and G.\,M. Ziegler:
\newblock{Basic Properties of Convex Polytopes.}
\newblock{CRC Press, Boca Raton 1997.}

\bibitem{Hosten2002}
S.~Ho{\c{s}}ten and S.~Sullivant:
\newblock{Gr\"obner bases and polyhedral geometry of reducible and cyclic
  models.}
\newblock{J. Combin. Theory Ser. A {\mi 100} (2002), 2, 277--301.}

\bibitem{kahle:2008}
T.~Kahle:
\newblock{Neighborliness of marginal polytopes.}
\newblock{Contrib. Algebra Geometry {\mi 51} (2010), 45--56.}

\bibitem{kahle:2006}
T.~Kahle and N.~Ay:
\newblock{Support sets of distributions with given interaction structure.}
\newblock{In: Proc. WUPES'06, 2006.}

\bibitem{kahle:2009}
T.~Kahle, W.~Wenzel, and N.~Ay:
\newblock{Hierarchical models, marginal polytopes, and linear codes.}
\newblock{Kybernetika {\mi 45} (2009), 189--208.}

\bibitem{Kalai:1993}
G.~Kalai:
\newblock{Some aspects of the combinatorial theory of convex polytopes.}
\newblock{1993.}

\bibitem{Kingman1978}
J.\,F.\,C. Kingman:
\newblock{Uses of exchangeability.}
\newblock{Ann. Probab. {\mi 6} (1978), 2, 183--197.}

\bibitem{Lindsay1995}
B.\,G. Lindsay:
\newblock{Mixture models: theory, geometry, and applications.}
\newblock{NSF-CBMS Regional Conference Series in Probability and Statistics.
  Institute of Mathematical Statistics, 1995.}

\bibitem{McLachlan:2000}
G.~McLachlan and D.~Peel:
\newblock{Finite Mixture Models.}
\newblock{Wiley Series in Probability and Statistics: Applied Probability and
  Statistics. Wiley, 2000.}

\bibitem{Montufar2011}
G.\,F. Mont\'ufar and N.~Ay:
\newblock{Refinements of universal approximation results for deep belief
  networks and restricted Boltzmann machines.}
\newblock{Neural Comput. {\mi 23} (2011), 5, 1306--1319.}

\bibitem{MonRauh2011}
G.\,F. Mont\'ufar, J.~Rauh, and N.~Ay:
\newblock{Expressive power and approximation errors of restricted Boltzmann
  machines.}
\newblock{In: Advances in Neural Information Processing
  Systems {\mi 24} (J.~Shawe-Taylor, R.~Zemel, P.~Bartlett, F.~Pereira, and
  K.~Weinberger, eds.), MIT Press, 2011, pp.\,415--423.}

\bibitem{Rauh11:Thesis}
J.~Rauh:
\newblock{Finding the Maximizers of the Information Divergence from an
Exponential Family.}
\newblock{Ph.\,D. Thesis, Universit\"at Leipzig, 2011.}

\bibitem{RKA10:Support_Sets_and_Or_Mat}
J.~Rauh, T.~Kahle, and N.~Ay:
\newblock{Support sets of exponential families and oriented matroids.}
\newblock{Internat. J. Approximate Reasoning {\mi 52} (2011), 5,
613--626.}

\bibitem{Settimi:1998}
R.~Settimi and J.\,Q. Smith:
\newblock{On the geometry of Bayesian graphical models with hidden variables.}
\newblock{In: Proc. Fourteenth conference on Uncertainty in
  artificial intelligence, UAI'98, Morgan Kaufmann Publishers
  1998, pp.\,472--479.}

\bibitem{Shemer:1982}
I.~Shemer.:
\newblock{Neighborly polytopes.}
\newblock{Israel J. Math. {\mi 43} (1982), 291--311.}

\bibitem{titterington:1985}
D.~Titterington, A.\,F.\,M. Smith, and U.\,E. Makov:
\newblock{Statistical Analysis of Finite Mixture Distributions.}
\newblock{John Wiley and Sons, 1985.}

\bibitem{Varshamov:1957}
R.~Varshamov:
\newblock{Estimate of the number of signals in error correcting codes.}
\newblock{Dokl. Akad. Nauk SSSR {\mi 117} (1957), 739--741.}

\end{thebibliography}
\def\P{\mathbb{P}}

\makecontacts

\end{document}